\documentclass{amsart}
\usepackage[svgnames]{xcolor}
\usepackage{amssymb}
\usepackage{tikz-cd}
\usepackage{enumitem}
\usepackage{mathtools}
\usepackage{hyperref}
\usepackage[capitalise]{cleveref}
\usepackage[T1]{fontenc}

\usepackage[old]{old-arrows}
\newcommand{\lrl}{%
\mathrel{\vcenter{\offinterlineskip \hbox{$\varleftarrow$}\vskip-.1ex\hbox{$\varrightarrow$}\vskip-.1ex\hbox{$\varleftarrow$}}}}

\newtheorem{theorem}{Theorem}[section] 
\newtheorem{thm}[theorem]{Theorem}
\newtheorem*{utheorem}{Theorem}
\newtheorem{cor}[theorem]{Corollary}
\newtheorem{prop}[theorem]{Proposition}
\newtheorem{lem}[theorem]{Lemma}
\theoremstyle{definition}
\newtheorem{definition}[theorem]{Definition}
\newtheorem{defn}[theorem]{Definition}
\theoremstyle{remark}
\newtheorem{remark}[theorem]{Remark}
\newtheorem{rmk}[theorem]{Remark}

\crefname{thm}{Theorem}{Theorems}

\newcommand{\cC}{\mathcal{C}}
\newcommand{\cD}{\mathcal{D}}

\newcommand{\cI}{{\mathcal{I}}}
\newcommand{\cJ}{\mathcal{J}}
\newcommand{\cK}{\mathcal{K}}
\newcommand{\cM}{\mathcal{M}}
\newcommand{\cN}{\mathcal{N}}
\newcommand{\cS}{\mathcal{S}}
\newcommand{\op}{\mathrm{op}}
\newcommand{\cat}{\cC\!\mathit{at}}
\newcommand{\set}{\cS\!\mathit{et}}
\newcommand{\sset}{\mathit{s}\set}
\newcommand{\ssset}{\mathit{ss}\set}
\newcommand{\cset}{\mathit{c}\set}
\newcommand{\msset}{\mathit{s}\set^{+}}
\newcommand{\catfin}{cat}
\newcommand{\pder}{\mathit{p}\cD\!\mathit{er}}
\DeclareMathOperator{\id}{id}
\DeclareMathOperator{\ho}{ho}
\newcommand{\lims}[1]{{#1}_\ell}
\newcommand{\rims}[1]{{#1}_r}
\newcommand{\sublurie}{{\mathrm{qcat}}}
\newcommand{\subrezk}{{\mathrm{css}}}
\newcommand{\subsattler}{{\mathrm{HoTT}}}
\newcommand{\subdkls}{{\mathrm{cJ}}}
\newcommand*\lcolon{\mathrel{:}}
\newcommand*\rcolon{\mathrel{:}}

\begin{document}

\author{Philip Hackney}
\address{Department of Mathematics,
University of Louisiana at Lafayette
}
\email{philip@phck.net} 

\author{Martina Rovelli}
\address{Department of Mathematics and Statistics, 
University of Massachusetts 
Amherst
}
\email{rovelli@math.umass.edu} 

\thanks{This material is based upon work supported by the National Science Foundation under Grant No.\ DMS-1440140 while the authors were in residence at the Mathematical Sciences Research Institute in Berkeley, California, during the Spring 2020 semester.}

\commby{}

\title{Induced model structures for higher categories}

\begin{abstract}
We give a new criterion guaranteeing existence of model structures left-induced along a functor admitting both adjoints. This works under the hypothesis that the functor induces idempotent adjunctions at the homotopy category level. As an application, we construct new model structures on cubical sets, prederivators, marked simplicial sets and simplicial spaces modeling $\infty$-categories and $\infty$-groupoids.
\end{abstract}

\maketitle

\section*{Introduction}

It is common when working in abstract homotopy theory to deal with several equivalent Quillen model categories, each with their own strengths and weaknesses. 
This even extends to Quillen model structures on a single category --- for instance, in categories of diagrams of a fixed shape in a model category $\cM$, there are often both projective and injective model structures, each of which have their place. 
It is in this spirit that we present the discovery of several new model structures for $\infty$-categories and $\infty$-groupoids. 
In particular, we give model structures supported on categories of cubical sets, prederivators, bisimplicial sets, and marked simplicial sets.  
For the purposes of the introduction, we focus on the first of these.

Cubical sets are presheaves on a category of cubes. 
But there are many possible categories of cubes, and there is a tension between the simplicity of the cube category and the expressivity of the corresponding category of cubical sets. 
We take this opportunity to point the reader to the introduction of \cite{CavalloMortbergSwan:UCMUTT} for an overview of what is known about various choices in the context of Homotopy Type Theory and Univalent Foundations.
In our case, we consider the category of cubes as a full subcategory of the category of posets, which is the same rich context that Kapulkin and Voevodsky operate under in \cite{KV}.  
This is a homotopically challenging indexing category to work with, as it does not admit a natural generalized Reedy structure.
Our first main result is \cref{cor cubical QE}, which is about the existence of induced model structures on cubical sets along the triangulation functor from the category of cubical sets to the category of simplicial sets. 
These model structures are (left- or right-) induced either from the Joyal or the Kan--Quillen model structure on the category of simplicial sets.
In all cases, the triangulation functor becomes both a left and a right Quillen equivalence.
On the other hand, in \cite{DohertyKapulkinLindseySattler:CMI1C}, a model structure for $(\infty,1)$-categories was given on a different category of cubical sets; in \cref{theorem comparison with dkls} we show that the natural comparison between the two categories of cubical sets is a Quillen equivalence between this one and the model structure left-induced from the Joyal model structure.

The main technical tools we use are lifting theorems for model structures in the presence of adjoint strings.
Let $\cM$ be a model category, $\cN$ be a (bicomplete) category, and suppose we have a string of adjoint functors
\begin{equation*}\label{adjoint string}
\begin{tikzcd}[column sep=large]
    \cN  \rar["F" description] 
    \rar[phantom, bend right=18, "\scriptscriptstyle\perp"]  
    \rar[phantom, bend left=18, "\scriptscriptstyle\perp"] 
    & \cM
    \lar[bend right=30, "L" swap] 
    \lar[bend left=30, "R"] 
\end{tikzcd}
\end{equation*}
(which below we will write more compactly as $F \lcolon \cN \lrl \cM \rcolon L,R$).
In \cite{DrummondColeHackney}, Drummond-Cole and the first author showed that if $\cM$ is cofibrantly generated and if the adjunction $FL\dashv FR$ on $\cM$ is a Quillen adjunction, then there exists a \emph{right-induced} model structure on $\cN$, where weak equivalences and fibrations in $\cN$ are created by $F$.
In that paper, the question was posed about whether one could guarantee a \emph{left-induced} model structure on $\cN$, that is, one where the weak equivalences and \emph{co}fibrations are created by $F$.
We give a partial answer.
\begin{utheorem}[Theorem~\ref{existenceleft} and Proposition~\ref{Quillenidempotent2}]\label{thm intro left induced}
Suppose that we have an adjoint string $F \lcolon \cN \lrl \cM \rcolon L,R$ with $\cN$ a locally presentable category and $\cM$ an accessible model category.
If $FL\dashv FR$ is a Quillen adjunction and the adjoint string between homotopy categories
$\ho\cN \lrl \ho\cM$
is an \emph{idempotent} adjoint string, then $\cN$ admits a model structure left-induced along $F$.
This occurs in the special case when $L$ or $R$ is fully faithful.
\end{utheorem}

In the theorem statement, the homotopy category $\ho\cN$ is obtained by inverting all morphisms which are sent by $F$ to weak equivalences in $\cM$.
An \emph{idempotent} adjoint string is an adjoint string where the two constituent adjunctions are idempotent adjunctions.

We should make clear that, just as the aforementioned Drummond-Cole--Hackney result rests on Kan's theorem for right-induced model structures \cite[Theorem 11.3.2]{Hirschhorn:MCL}, so too does our present theorem rest on the the Acyclicity Theorem of \cite{GarnerKedziorekRiehl,HKRS}.

\subsection*{Acknowledgements}
We would like to thank Gabriel C.~Drummond-Cole, Richard Garner, Chris Kapulkin, Viktoriya Ozornova, Emily Riehl and Christian Sattler for useful input, suggestions, and encouragement.

\section{The abstract framework and results}

\subsection{Strings of adjoint functors}
We start by recalling the terminology related to strings of adjoint functors.
A pair of adjoint functors with $F$ left adjoint to $G$ will be denoted by either $F \lcolon \cM \rightleftarrows \cN \rcolon G$ or $F \dashv G$, depending on whether or not we wish to emphasize the (co)domains of the functors.

\begin{defn}
Let $\cM$ and $\cN$ be categories. A \emph{string of adjoint functors} (or \emph{adjoint string}) $F \lcolon \cN \lrl \cM \rcolon L,R $ consists of functors $F\colon\cN\to\cM$ and $L,R\colon\cM\to\cN$ that form adjunctions $L\dashv F$ and $F\dashv R$.
\end{defn}

\begin{rmk}
Given any string of adjoint functors $F \lcolon \cN \lrl \cM \rcolon L,R $, the adjunctions
$F\lcolon\cN\rightleftarrows\cM\rcolon R$ and $L\lcolon\cM\rightleftarrows\cN\rcolon F$
can be composed to obtain an adjunction
$F L\lcolon\cM\rightleftarrows\cM\rcolon F R$.
\end{rmk}

The following gives a large source of examples of strings of adjoint functors.

\begin{rmk}\label{remark left right kan extension adjoint triple}
Let $\cS$ be a bicomplete category. 
Any functor $f\colon \cI\to \cJ$ between small categories induces a string of adjoint functors $f^*\lcolon\cS^{\cJ}\lrl\cS^{\cI}\rcolon f_!,f_*$ where $f^*\colon\cS^{\cJ}\to\cS^\cI$ denotes the restriction functor and $f_!\colon\cS^{\cI}\to\cS^{\cJ}$ (resp.\ $f_*\colon\cS^{\cI}\to\cS^{\cJ}$) denotes the left (resp.\ right) Kan extension along $f$.
\end{rmk}

We consider the following two properties that adjoint strings of functors may have.
The first is studied more deeply in \cite[\textsection2]{Johnstone:RPLC}.

\begin{defn}
\label{characterizationidempotent}
Let $F\lcolon \cN \lrl \cM \rcolon L,R$ be a string of adjoint functors, let $\eta$ and $\epsilon$ be the unit and counit of the adjunction $L\dashv F$, let $\eta'$ and $\epsilon'$ be the unit and counit of the adjunction $F\dashv R$.
We say that the adjoint string is \emph{idempotent} if all of the maps in the following diagram are isomorphisms.
\[ \begin{tikzcd}[column sep=small]
& FLF \ar[dr,Rightarrow,"F\epsilon"] \\
F \ar[rr,Rightarrow,"\id_F"] \ar[dr, Rightarrow,"F\eta'"'] \ar[ur,Rightarrow,"\eta F"] & & F \\
& FRF \ar[ur,Rightarrow,"\epsilon'F"']
\end{tikzcd} \]
This happens if and only if one of the outside maps is an isomorphism.
\end{defn}

If $F$ happens to be fully faithful, then $F\epsilon$ is an isomorphism by \cite[Theorem IV.3.1]{MacLane}, hence the string is idempotent
(if $F$ is conservative then the converse also holds).
Another, different, special case of idempotency is that of a `fully faithful' string of adjoint functors, which we now recall. 

\begin{defn}
\label{characterizationfullyfaithful}
A string of adjoint functors $F\lcolon \cN \lrl \cM \rcolon L,R $ is said to be \emph{fully faithful} if both $L$ and $R$ are fully faithful.
If one of $L$ or $R$ is fully faithful, so is the other (see \cite[Lemma~1.3]{DyckhoffTholen:EMPPPC}).
\end{defn}

\begin{rmk}
Any fully faithful string of adjoint functors $F \lcolon \cN \lrl \cM \rcolon L,R $ is an idempotent string of adjoint functors.
This is because $L$ being fully faithful is equivalent to the unit $\eta\colon \id_{\cM}\Rightarrow FL$ of the adjunction $L\dashv F$ being an isomorphism  (alternatively, $R$ being fully faithful is equivalent to the counit $\epsilon'$ of the adjunction $F \dashv R$ being an isomorphism) by \cite[Theorem IV.3.1]{MacLane}.
\end{rmk}

\begin{rmk}
\label{adjointtriplefromchangeofshape}
Let $\cS$ be a bicomplete category. If $f\colon\cI\to\cJ$ is a fully faithful functor between small categories, then the string of adjoint functors $f^*\lcolon\cS^{\cJ}\lrl\cS^{\cI}\rcolon f_!,f_*$ from Remark~\ref{remark left right kan extension adjoint triple} is fully faithful.
\end{rmk}

\subsection{Model structures induced via adjoint strings}

We discuss situations in which one can transfer a model structure along the middle functor of a string of adjoint functors. 
In this paper, model categories will admit all small limits and colimits (which we refer to as `bicomplete') and will be assumed to come equipped with functorial factorizations.
In particular, a model category $\cM$ comes equipped with a natural transformation $(-)^c \Rightarrow \id_{\cM}$ so that each component $X^c \to X$ is an acyclic fibration from a cofibrant object (and dually for $\id_{\cM} \Rightarrow (-)^f$).
The term `left Quillen functor' will be synonymous with `left adjoint in a Quillen adjunction', and `left Quillen equivalence' will mean `left adjoint in a Quillen equivalence' (and similarly for right adjoints).
We will often be interested in the accessible model categories of \cite[Definition 3.1.6]{HKRS}, which are a generalization of Jeff Smith's combinatorial model categories (see, e.g., \cite[\S A.2.6]{htt}).

\begin{defn}
Let $F\colon\cN\to\cM$ be a functor, and $\cM$ a model category. The \emph{left-induced model structure} $\lims{\cN}$ (resp.~\emph{right-induced model structure} $\rims{\cN}$) on $\cN$, when it exists, is the model structure in which the cofibrations (resp.~the fibrations) and the weak equivalences are created by $F$.
\end{defn}

The following criterion for existence of right-induced model structure along $F$ in the presence of a string of adjoint functors $L\dashv F\dashv R$ was given by Drummond-Cole and the first author.

\begin{thm}[{\cite[Theorem~2.3]{DrummondColeHackney}}]
\label{existenceright}
Let $F \lcolon \cN \lrl \cM \rcolon L,R $ be a string of adjoint functors, where $\cM$ is a cofibrantly generated model category and $\cN$ is a bicomplete category. 
If $F L\lcolon\cM\rightleftarrows\cM \rcolon F R$ is a Quillen adjunction, then the category $\cN$ supports the right-induced model structure $\rims{\cN}$. 
Further, the functor $F\colon\rims{\cN}\to\cM$ is both left and right Quillen.
\end{thm}

Under the additional hypothesis that $F$ is fully faithful, Campbell observed in \cite[Proposition 2.2]{Campbell:JCC} that the model structure on $\cN$ is also left-induced.

The following makes no reference to the functor $R$, but we state it this way as it is most meaningful in the case when $FL \dashv FR$ is a Quillen adjunction on $\cM$.

\begin{definition}[Homotopy idempotent string]\label{def homotopy idempotent}
Let $\cM$ be a model category and let $F\lcolon \cN \lrl \cM \rcolon L,R $ be a string of adjoint functors.
\begin{itemize}[leftmargin=*]
\item For each $X$ in $\cN$, define a map $\epsilon_X^h$ of $\cN$ as the composite
\[
    \epsilon_X^h\colon L((FX)^c) \to LFX \to X
\]
whose first map is $L$ applied to the cofibrant replacement $(FX)^c \overset\sim\to FX$ in $\cM$, and whose second map is the counit of $L\dashv F$.
Notice that $\epsilon^h \colon L((F-)^c) \Rightarrow \id_{\cN}$ is a natural transformation.
\item We say that the string of adjoint functors is \emph{homotopy idempotent} if the map  $F\epsilon_X^h \colon FL((FX)^c) \to FX$ is a weak equivalence in $\cM$ for every object $X$ in $\cN$.
\end{itemize}
\end{definition}

\begin{remark}\label{remark derived unit}
We have chosen this definition because it is the most convenient for the proof of Theorem~\ref{existenceleft}, but of course there are other formulations (see also Proposition~\ref{Quillenidempotent2} below).
For example, by combining the naturality square for the unit $\eta$ of the adjunction $L\dashv F$ with a triangle identity, we obtain the following commutative diagram in $\cM$:
\[ \begin{tikzcd}
(FX)^c \dar{\eta_{(FX)^c}} \rar[two heads, "\sim"] & FX \dar{\eta_{FX}} \arrow[dr, bend left=15, "="' swap] \\
FL((FX)^c) \rar  & FLFX \rar["F\epsilon_X"] & FX.
\end{tikzcd} \]
By two-of-three, one sees that $F\epsilon_X^h$ is a weak equivalence in $\cM$ if and only if the map $\eta_{(FX)^c}\colon(FX)^c\to FL((FX)^c)$ is a weak equivalence.
\end{remark}

\begin{rmk}
A string of adjoint functors $F \lcolon \cN \lrl \cM \rcolon L,R $ is automatically
a homotopy idempotent string in the following cases:
\begin{itemize}[leftmargin=*]
    \item when the adjoint string $L\dashv F\dashv R$ is idempotent, the adjunction $FL\dashv FR$ is a Quillen pair and every object in $\cM$ is cofibrant, or
    \item when the adjoint string $L\dashv F\dashv R$ is fully faithful.
\end{itemize}
\end{rmk}

We now give a criterion for the existence of the left-induced model structure along $F$ in presence of a string of adjoint functors $L\dashv F\dashv R$.

\begin{thm}
\label{existenceleft}
Let $F \lcolon \cN \lrl \cM \rcolon L,R $ be a string of adjoint functors, where $\cM$ is an accessible model category and $\cN$ is a locally presentable category. 
If $FL \dashv FR$ is a Quillen adjunction and $F\lcolon \cN \lrl \cM \rcolon L,R $ is a homotopy idempotent string of adjoint functors, then the category $\cN$ supports the left-induced model structure $\lims{\cN}$.
Further, the functor $F\colon\lims{\cN}\to\cM$ is both left and right Quillen.
\end{thm}

The proof of the preceding theorem uses the cylinder object argument, whose statement we recall here for the convenience of the reader. 
The conclusion of Theorem~\ref{cylinderobject} relies on the Garner--Hess--K\k{e}dziorek--Riehl--Shipley Acyclicity Theorem, whose corrected proof may be found in \cite{GarnerKedziorekRiehl}.

\begin{thm}[{\cite[Theorem 2.2.1]{HKRS}}]
\label{cylinderobject}
Let $F\lcolon\cN\rightleftarrows\cM\rcolon G$ be an adjoint pair, $\cM$ an accessible model category and $\cN$ a locally presentable category. Suppose that the following conditions hold.
\begin{enumerate}[leftmargin=*]
\item For every object $X$ of $\mathcal{N}$, there exists a morphism $\phi_X \colon QX \to X$ such that $F\phi_{X}$ is a weak equivalence and $F(QX)$ is cofibrant in $\mathcal{M}$,  \label{cylinder item number one}
\item For each morphism $f\colon X \to Y$ in $\mathcal{N}$  there exists a morphism $Qf\colon QX \to QY$ satisfying $\phi_{Y}\circ Qf =f\circ \phi_{X}$, and \label{cylinder item number two}
\item For every object $X$ of $\mathcal{N}$, there exists a factorization $QX \amalg QX \xrightarrow j\mathrm{Cyl}(QX) \xrightarrow p QX$  of the fold map such that $Fj$ is a cofibration and $Fp$ is a weak equivalence.\label{cylinder item number three}
\end{enumerate}
Then $\cN$ supports the model structure $\lims{\cN}$ left-induced along $F$.
\end{thm}

The following proof includes a verification of the conditions of \cref{cylinderobject}.

\begin{proof}[Proof of \cref{existenceleft}]
For an object $X$ of $\cN$, we define $QX$ to be $L((FX)^c)$ and let $\phi_X$ denote the map $\epsilon_X^h$ from \cref{def homotopy idempotent}:
\[ \phi_X \coloneqq \epsilon^h_X \colon QX = L((FX)^c) \xrightarrow{L(\overset\sim\twoheadrightarrow)} LFX \xrightarrow{\epsilon_X} X.\]
We now verify the conditions to apply \cref{cylinderobject}.
The map $F\epsilon^h_X$ is a weak equivalence in $\cM$ by assumption, and, since $FL$ is left Quillen, $FQX=FL((FX)^c)$ is a cofibrant object of $\mathcal{M}$. Thus we have established Condition \eqref{cylinder item number one}.
Condition \eqref{cylinder item number two} holds because $\epsilon^h$ is a natural transformation $Q\Rightarrow \id_{\cN}$.

For Condition \eqref{cylinder item number three}, notice that we have a factorization
\[
\begin{tikzcd}
	(FX)^c \amalg (FX)^c \rar[rightarrowtail, "J"] & \mathrm{Cyl}((FX)^c) \rar[two heads, "\sim" swap, "P"] & (FX)^c
\end{tikzcd}
\]
of the fold map of $(FX)^c$ in $\mathcal{M}$.
The objects $\mathrm{Cyl}((FX)^c)$ and $(FX)^c$ are cofibrant in $\cM$, and by Ken Brown's lemma \cite[Corollary 7.7.2(1)]{Hirschhorn:MCL} the left Quillen functor $FL$ preserves weak equivalences between cofibrant objects, hence $FLP$ is a weak equivalence in $\cM$.
Further, $FLJ$ is again a cofibration since $FL$ is left Quillen.
Thus the factorization 
\[
\begin{tikzcd}
	QX \amalg QX \cong L\left[(FX)^c \amalg (FX)^c\right] \rar["LJ"] & L\mathrm{Cyl}((FX)^c) \rar["LP"] & QX
\end{tikzcd}
\]
of the fold map of $QX$ has the desired properties.

We conclude from \cref{cylinderobject} that the left-induced model structure exists. The functor $F$ is left Quillen by construction, and the proof that $F$ is right Quillen is the evident variation from the one appearing in \cite[Theorem 2.3]{DrummondColeHackney} and does not depend on homotopy idempotency.
That is, one uses that $FL$ preserves (acyclic) cofibrations and $F$ reflects them to infer that $L$ is left Quillen.
\end{proof}

\begin{rmk}
A dual version of \cite[Theorem 5.6]{DrummondColeHackney} holds, which allows one to efficiently lift Quillen equivalences in circumstances where Theorem~\ref{existenceleft} holds.
As we will not need this result here we will not repeat the statement (the proof is formally dual), and instead merely note that the last two bullet points of \cite[Theorem 5.6]{DrummondColeHackney} should be replaced with `$F'$ reflects cofibrations and preserves fibrant objects' and `$F$ preserves cofibrations and cofibrant objects', respectively.
\end{rmk}

By the dual of \cite[Proposition 2.4]{DrummondColeHackney}, if $\cM$ is left or right proper, then the same is true for $\lims{\cN}$.
The following proposition guarantees that the left- and right-induced model structures, if they exist, have equivalent homotopy theories.

\begin{prop}\label{prop comparison of left and right}
Let $F\lcolon \cN \lrl \cM \rcolon L,R $ be a string of adjoint functors, and $\cM$ a model category. If the left-induced model structure $\lims{\cN}$ and the right-induced model structure $\rims{\cN}$ both exist,
then the adjunction
    $\id_{\cN}\lcolon\rims{\cN}\rightleftarrows\lims{\cN}\rcolon \id_{\cN}$
    is a Quillen equivalence.
\end{prop}
\begin{proof}
As the two model structures share the same class of weak equivalences, it is enough to show that $\id_{\cN} \colon \lims{\cN} \to \rims{\cN}$ preserves fibrations.
This holds because $F\colon \lims{\cN} \to \cM$ preserves fibrations and $F\colon \rims{\cN} \to \cM$ reflects them.
\end{proof}

\begin{cor}
\label{corollary fully faithful string}
Let $F\lcolon \cN \lrl \cM \rcolon L,R $ be a fully faithful string of adjoint functors, $\cM$ a combinatorial model category and $\cN$ a locally presentable category.
Then the category $\cN$ admits both the left-induced model structure $\lims{\cN}$ and the right-induced model structure $\rims{\cN}$.
Further, the left diagram below is a diagram of left Quillen equivalences, and the right diagram below is a diagram of right Quillen equivalences.
\[ \begin{tikzcd}[column sep=small]
\rims{\cN} \ar[rr,"\id_{\cN}"] \ar[dr,"F"'] & & \lims{\cN} \ar[dl, "F"] 
&
\rims{\cN} \ar[dr,"F"'] & & \lims{\cN} \ar[dl, "F"]  \ar[ll,"\id_{\cN}"']
\\
& \cM & & & \cM
\end{tikzcd} \]
\end{cor}

\begin{proof}
As $\cM$ is combinatorial, it is both accessible and cofibrantly generated.
We assumed that $L$ is fully faithful, implying that the functor $FL\cong \id_{\cM}$ is left Quillen, so both Theorem~\ref{existenceleft} and Theorem~\ref{existenceright} apply to show the existence of the indicated model structures.
By \cite[Theorem 3.2]{FKKR}, the functor $F \colon \rims{\cN} \to \cM$ is both a left and right Quillen equivalence.
The fact that $F \colon \lims{\cN} \to \cM$ is both a left and a right Quillen equivalence now follows from \cref{prop comparison of left and right} and two out of three for Quillen equivalences \cite[Corollary 1.3.15]{Hovey:MC}.
\end{proof}

We devote the remainder of this section to Proposition~\ref{Quillenidempotent2}, which gives an interpretation of what being a homotopy idempotent adjoint string (in the sense of \cref{def homotopy idempotent}) means at the level of homotopy categories.
To give the most natural statement of this proposition, we make use of the language of deformable adjunctions from \cite[\S2.2]{RiehlCHT} (originally from \cite[\textsection44.2]{DwyerHirschhornKanSmith}), which will not be needed elsewhere.
The reader who prefers to stay within the world of Quillen model categories should be comforted to know that with some adjustments this is possible, and such a reader is advised to immediately skip down to Remark~\ref{remark stay within Quillen world} and the statement of Proposition~\ref{Quillenidempotent2}.

\begin{rmk}
\label{adjunctionhomotopycategory}
Let $F\lcolon \cN \lrl \cM \rcolon L,R $ be a string of adjoint functors and $\cM$ a model category such that $F L\lcolon\cM\rightleftarrows\cM\rcolon F R$ is a Quillen pair.
We endow $\cN$ with the class of weak equivalences created by $F$, and denote by $\ho\cN$ the corresponding homotopy category (which is potentially not locally small). 
In this situation, one can then show that $F\colon\cN\to\cM$ is homotopical, while $L\colon\cM\to\cN$ (resp.~$R\colon\cM\to\cN$) preserves weak equivalences between cofibrant (resp.~fibrant) objects, by Ken Brown's lemma \cite[Corollary 7.7.2(1)]{Hirschhorn:MCL}. 
In particular, the adjunctions $L\lcolon\cM\rightleftarrows\cN\rcolon F$ and $F\lcolon\cN\rightleftarrows\cM\rcolon R$ are \emph{deformable adjunctions}, and by \cite[Theorem~2.2.11]{RiehlCHT} they induce adjunctions $\overline L\lcolon\ho\cM\rightleftarrows\ho\cN\rcolon\overline F$ and $\overline F\lcolon\ho\cN\rightleftarrows\ho\cM\rcolon\overline R$ at the level of homotopy categories, where the values of the functors on objects are $\overline LA = L(A^c)$, $\overline RA=R(A^f)$ and $\overline FX = FX$.
There is a string of adjoint functors at the level of homotopy categories $\overline F\lcolon \ho\cN \lrl \ho\cM \rcolon \overline L,\overline R$.
\end{rmk}

\begin{prop}
\label{Quillenidempotent2}
Let $F\lcolon \cN \lrl \cM \rcolon L,R $ be a string of adjoint functors, and $\cM$ a model category such that $F L\lcolon\cM\rightleftarrows\cM\rcolon F R$ is a Quillen pair. 
When $\cN$ is endowed with the class of weak equivalences created by $F$, the following are equivalent.
\begin{enumerate}[leftmargin=*]
\item The functor $\overline F \colon \ho\cN \to \ho\cM$ is fully faithful.
\item The string $\overline F \lcolon \ho\cN \lrl \ho\cM \rcolon \overline L,\overline R $ is an idempotent adjoint string.\label{item idempotent at homotopy level}
\item The string $F\lcolon \cN \lrl \cM \rcolon L,R $ is a homotopy idempotent adjoint string.\label{item homotopy idempotent}
\end{enumerate}
\end{prop}

\begin{proof}
As observed in \cref{adjunctionhomotopycategory}, the adjunction $L\lcolon\cM\rightleftarrows \cN\rcolon F$ induces an adjunction $\overline{L}\lcolon\ho\cM\rightleftarrows \ho\cN\rcolon \overline{F}$ at the level of homotopy categories.
The natural bijections coming from the total derived adjunction $\overline L \dashv \overline F$ and the adjunction $L\dashv F$ fit into the square
\[
\begin{tikzcd}[row sep=small]
\cN(LA,X) \dar \rar[leftrightarrow,"\cong"] & \cM(A,FX) \dar \\
\ho\cN(LA,X) \dar & \ho\cM(A,FX) \dar{=} \\
\ho\cN(\overline LA,X) \rar[leftrightarrow,"\cong"] & \ho\cM(A,\overline FX).
\end{tikzcd}
\]

One sees that the counit $\overline \epsilon \colon \overline L \phantom{.}\overline F \Rightarrow \id_{\ho \cN}$ of the adjunction $\overline{L}\dashv\overline{F}$ is represented at an object $X$ by the map $\epsilon_X^h \colon L (FX)^c \to LFX \to X$ in $\cN$, obtained by applying $L$ to the cofibrant replacement map of $FX$ and by composing with $\epsilon_X$, the counit of the adjunction $L\dashv F$.
As $F$ is homotopical, the natural transformation $\overline F \overline \epsilon \colon \overline F \phantom{.}\overline L\phantom{.} \overline F \Rightarrow \overline F$ at $X$ is then represented by $F\epsilon_X^h\colon FL (FX)^c \to FLFX \to FX$ in $\cM$. 
Since $F$ creates weak equivalences, $\epsilon_X^h$ is weak equivalence if and only if $F\epsilon_X^h$ is. 
The equivalence now follows from \cref{characterizationidempotent} and \cref{def homotopy idempotent}.
\end{proof}

\begin{remark}\label{remark stay within Quillen world}
Suppose we are in the situation of Proposition~\ref{Quillenidempotent2}, and additionally assume that $\cM$ is cofibrantly generated and that $\cN$ is bicomplete.
One then has access to the right-induced model structure $\rims{\cN}$ from Theorem~\ref{existenceright}, so that the functor $F \colon \rims{\cN} \to \cM$ is both left and right Quillen.
Then the category $\ho\cN = \ho\rims{\cN}$ is defined as usual, and the induced adjoint string at the level at the level of homotopy categories uses that the left and right derived functors of $F$ coincide (see \cite[Corollary 7.8]{Shulman:CCLRDF}).
In that special case, this proposition may be proved by explicitly comparing the map from \cref{def homotopy idempotent} to the derived unit of the Quillen pair $L\dashv F$.
\end{remark}

\begin{remark}[Bousfield--Friedlander (co)localization]
Suppose $\cM$ and $\cN$ are model categories and $F \lcolon \cN \lrl \cM \rcolon L,R$ is a homotopy idempotent string where $F$ creates weak equivalences.
It would be interesting to investigate further conditions on the adjoint string which guarantee that the Bousfield--Friedlander localization $\cM^{FL}$ of $\cM$ at the $FL$-equivalences exists \cite{MR2427416}, since one can prove that $F\colon \cN \to \cM^{FL}$ is a right Quillen equivalence.
Dually, $F$ will give a left Quillen equivalence from $\cN$ to the colocalization of $\cM$ at the $FR$-equivalences \cite[\S3]{MR2277698}, provided it exists.
\end{remark}

\section{Cubical sets}

In this section and the next, we denote by $\sset_{(\infty,0)}$ the Kan--Quillen model structure on the category $\sset$ for $(\infty,0)$-categories (namely $\infty$-groupoids), and by $\sset_{(\infty,1)}$ the Joyal model structure on $\sset$ for $(\infty,1)$-categories (see e.g.~\cite{DuggerSpivakMapping,htt}). In these model structures, the cofibrations are precisely the monomorphisms and the fibrant objects are respectively the Kan complexes and the quasi-categories.
We uniformly denote these two model structures by $\sset_{(\infty,\varepsilon)}$, where $\varepsilon=0,1$.

The goal of the present section is to discuss several model structures on the category of cubical sets.
As mentioned in the introduction, there are many useful categories of cubical sets, depending on the choice of the underlying cube category; several of these are discussed in \cite{GrandisMauri,BuchholtzMorehouse}.
For example, one can use the minimal structure where only the face and degeneracy maps are present, or one could add in either positive or negative connections, or one can consider all poset maps between cubes.
These cube categories are Grothendieck test categories, so each is suitable for modeling $\infty$-groupoids (see  \cite{Jardine:CHT,cisinski,Maltsiniotis:CCCUCT,StreicherWeinberger}).
More recently, model structures for higher categories have been developed for cubical sets with connection: for $\infty$-categories this was done in \cite{DohertyKapulkinLindseySattler:CMI1C}, and using marked cubical sets a model for $(\infty,n)$-categories was given in \cite{CampionKapulkinMaehara:ACMINC}.
The methods of the last references break down when working with the full cube category (which is not even a generalized Reedy category in the sense of \cite{BergerMoerdijk:OENRC}), but below we will show that one can nevertheless obtain interesting results.
We compare one of our new model structures to a type-theoretical model structure from \cite{SattlerIdempotent,StreicherWeinberger} and another to the cubical Joyal model structure of \cite{DohertyKapulkinLindseySattler:CMI1C}.

\subsection{New model structures on cubical sets}\label{subsec cube}
Let $\square$ denote the full subcategory of $\cat$ of cubes $[1]^n$ for $n\ge0$,
and let $\cset=\set^{\square^{\op}}$ denote the category of \emph{cubical sets}, namely presehaves $X\colon\square^{\op}\to\set$. 
This category of cubes, and the corresponding  category of cubical sets, are those studied in \cite{SattlerIdempotent, KV} (see \cite{CCHM} for a closely related cube category).

There is a triangulation functor $T\colon\cset\to\sset$, which can be defined on representables by $T(\square[1]^n) \coloneqq \Delta[1]^n$, and extended cocontinuously to all cubical sets, and the functor $T$ has a right adjoint $C\colon\sset\to\cset$. 
Kapulkin--Voevodsky show in \cite[\textsection1]{KV} that the functor $C$ is fully faithful.

Sattler shows in \cite[Theorem 2.1]{SattlerIdempotent} that for this specific choice of category of cubical sets the functor $T$ also admits a left adjoint $L\colon\sset\to\cset$. In particular, there is a fully faithful adjoint string
\[T\lcolon\cset\lrl\sset\rcolon L,C\]
between the category of simplicial sets and this specific choice of category of cubical sets.

From \cref{corollary fully faithful string}, we obtain the following, which endows $\cset$ with two Quillen equivalent model structures for $(\infty,\varepsilon)$-categories for any fixed $\varepsilon=0,1$.

\begin{prop}\label{cor cubical QE}
Let $\varepsilon=0,1$.
The category $\cset$ admits both the left-induced model structure $\cset_{\ell,\varepsilon}$ and the right-induced model structure $\cset_{r,\varepsilon}$ along $T\colon\cset\to\sset_{(\infty,\varepsilon)}$. 
Further, the left diagram below is a diagram of left Quillen equivalences, and the right diagram below is a diagram of right Quillen equivalences.
\[ \begin{tikzcd}[column sep=tiny]
\cset_{r,\varepsilon}\ar[rr,"\id_{\cset}"] \ar[dr,"T"'] & & \cset_{\ell,\varepsilon} \ar[dl, "T"] 
&
\cset_{r,\varepsilon} \ar[dr,"T"'] & & \cset_{\ell,\varepsilon} \ar[dl, "T"]  \ar[ll,"\id_{\cset}"']
\\
& \sset_{(\infty,\varepsilon)} & & & \sset_{(\infty,\varepsilon)}
\end{tikzcd} \]
\end{prop}
\begin{proof}
As mentioned above, $T$ admits both adjoints.
Further, as discussed above, the right adjoint $C$ was shown to be fully faithful by Kapulkin--Voevodsky in \cite[\S1]{KV}, so this is a fully faithful adjoint string.  
Corollary~\ref{corollary fully faithful string} implies the result.
\end{proof}

    The following lemma, which we learned from Sattler, clarifies the relation between the 
classes of cofibrations in the model structures $\cset_{r,\varepsilon}$ and $\cset_{\ell,\varepsilon}$.

    \begin{lem}\label{lemma monos cofib}
    The class of monomorphisms of $\cset$ contains the cofibrations of $\cset_{r,\varepsilon}$
    and is contained in the class of cofibrations of $\cset_{\ell,\varepsilon}$.
    \end{lem}
    
    \begin{proof}
    The generating cofibrations of $\cset_{r,\varepsilon}$ are given by applying $L$ to a set of generating cofibrations for $\sset_{(\infty,\epsilon)}$, and 
    Sattler shows in \cite[Proposition~3.3]{SattlerIdempotent} that the functor $L$ preserves monomorphisms. Since $\cset$ is a presheaf category, the class of monomorphisms of $\cset$ is saturated. It follows that the class of cofibrations of $\cset_{r,\varepsilon}$ is contained in the class of monomorphisms of $\cset$.
    
    Further, being a right adjoint the functor $T$ preserves monomorphisms.
It follows, using the description of the cofibrations in $\cset_{\ell,\varepsilon}$, that the class of monomorphisms of cubical sets is contained in the class of cofibrations of $\cset_{\ell,\varepsilon}$.
\end{proof}

\begin{prop}
    \label{corcsets}
Let $\varepsilon=0,1$. 
There is a model structure on cubical sets, denoted by $\cset_{m,\varepsilon}$, in which the cofibrations are the monomorphisms and the weak equivalences are created by $T\colon\cset\to\sset_{(\infty,\varepsilon)}$. Further, the left diagram below is a diagram of left Quillen equivalences, and the right diagram below is a diagram of right Quillen equivalences.
\[ \begin{tikzcd}[column sep=small]
\cset_{r,\varepsilon}\rar["\id"] \ar[dr,"T"'] & \cset_{m,\varepsilon} \dar["T"] \rar["\id"] & \cset_{\ell,\varepsilon} \ar[dl, "T"]
\\
& \sset_{(\infty,\varepsilon)} & 
\end{tikzcd}
\quad
\begin{tikzcd}[column sep=small]
\cset_{r,\varepsilon} \ar[dr,"T"'] &\lar["\id",swap] \cset_{m,\varepsilon} \dar["T"]  & \lar["\id",swap] \cset_{\ell,\varepsilon} \ar[dl, "T"]
\\
& \sset_{(\infty,\varepsilon)} & 
\end{tikzcd}
\]
\end{prop}

\begin{proof}
By a theorem of Jardine \cite{JardineIntermediate} (following the interpretation from \cite[Remark 2.3.4]{HKRS}), we know that if $\cN$ is a category endowed with two model structures $\cN_1$ and $\cN_2$ having the same class of weak equivalences, and if $\cK$ is a class of maps which contains the cofibrations of $\cN_1$ and is contained in the cofibrations of $\cN_2$, then $\cN$ supports a model structure in which the class of weak equivalences is the same as for $\cN_1$ and $\cN_2$, and the class of cofibrations is $\cK$.
By \cref{lemma monos cofib}, we can apply this when $\cK$ is the class of monomorphisms of $\cset$ to obtain the indicated model structures.
\end{proof}

The model structure $\cset_{m,0}$ should be related to the \emph{test model structure} (see \cite[Remark 1.2]{KV}) which has the same cofibrations and also has weak equivalences created in $\infty$-groupoids.
The model structures $\cset_{r,\varepsilon}$ for $\varepsilon=0,1$ were known to Sattler, while the model structures $\cset_{\ell,\varepsilon}$ for $\varepsilon=0,1$, as well as the model structure $\cset_{m,1}$ are new.

\subsection{A comparison with cubical models of homotopy type theory}

The other model structure on the category $\cset$ considered in the literature, motivated by homotopy type theory, is the minimal Cisinski model structure on $\cset$ from \cite[\textsection3]{StreicherWeinberger}, which was also studied in \cite[\textsection3.3]{SattlerIdempotent}.
    
\begin{thm}
There is a model structure on $c\set$, which we denote by $\cset_\subsattler$, in which the cofibrations are the monomorphisms, and the fibrations are the maps that have the right lifting property with respect to the class of maps
\[(\square[0]\times B)\amalg_{\square[0]\times A}(\square[1]\times A)\to\square[1]\times B\]
where $A\to B$ is a monomorphism of cubical sets and $\square[0]\to\square[1]$ is either of the two canonical inclusions.
Further, the functor $T\colon\cset_\subsattler\to\sset_{(\infty,0)}$ is left and right Quillen.
\end{thm}

It is mentioned in \cite{Sattler:DCMTTMHT} and in \cite[\textsection5]{StreicherWeinberger}  that, unlike for other categories of cubical sets, it is an open problem whether the homotopy theoretic model structure $\cset_\subsattler$ is Quillen equivalent to the model structure $\sset_{(\infty,0)}$ for $\infty$-groupoids.
We now explore how the model structure $\cset_\subsattler$ compares with $\cset_{m,0}$.

\begin{prop}
\label{remark quillen pair hott}
The identity is a left Quillen functor $\cset_\subsattler \to \cset_{m,0}$.
\end{prop}

\begin{proof}
The two model structures have the same class of cofibrations, so it suffices to show that pushout-products 
\[(\square[0]\times B)\amalg_{\square[0]\times A}(\square[1]\times A)\to\square[1]\times B\]
of monomorphisms $A\hookrightarrow B$ and inclusions $\square[0]\hookrightarrow\square[1]$ of $\cset_\subsattler$ are weak equivalences in $\cset_{m,0}$, namely are sent by $T$ to weak equivalences of $\sset_{(\infty,0)}$. 
Since the functor $T$ is both a left and a right adjoint, it preserves pushouts, products and monomorphisms, and therefore it sends this map to the pushout product in $\sset$
\[(\Delta[0]\times TB)\amalg_{\Delta[0]\times TA}(\Delta[1]\times TA)\to\Delta[1]\times TB\]
of the monomorphism $TA\hookrightarrow TB$ (which is a cofibration of $\sset_{(\varepsilon,0)}$) and the inclusion $\Delta[0]\hookrightarrow\Delta[1]$ (which is an acyclic cofibration of $\sset_{(\varepsilon,0)}$). Since $\sset_{(\infty,0)}$ is a cartesian closed model category, the desired map is a weak equivalence of $\sset_{(\infty,0)}$.
\end{proof}

However, understanding whether the identity functor from \cref{remark quillen pair hott} is a left Quillen equivalence\footnote{As these model structures share the same class of cofibrations, this is a left Quillen equivalence if and only if the two model structures are equal.} is a non-trivial matter. Indeed, it is equivalent to the functor $T \colon \cset_\subsattler \to \sset_{(\infty,0)}$ being a left Quillen equivalence, which as mentioned earlier is an open question.

\subsection{A comparison between cubical models for \texorpdfstring{$(\infty,1)$}{(∞,1)}-categories}

In \cite{DohertyKapulkinLindseySattler:CMI1C}, a different model structure was constructed on cubical sets which also serves as a model for $(\infty,1)$-categories.
They use smaller cube categories than we have used above, and they give a much more explicit description of their model structure than we have given; this process is helped along by the fact that their cube categories are EZ-Reedy categories in the sense of \cite{elegant}.

For concreteness, write $\square'$ for the wide subcategory of $\square$ which is generated by faces, degeneracies, and both positive and negative connections \cite[1.2]{DohertyKapulkinLindseySattler:CMI1C}, and let $k\colon \square' \to \square$ be the inclusion functor.
The main theorem of \cite{DohertyKapulkinLindseySattler:CMI1C} is that there is a Quillen model structure $\cset'_\subdkls$, dubbed the \emph{cubical Joyal model structure}, on $\cset' \coloneqq \set^{(\square')^\op}$ so that the triangulation functor $T' \colon \cset'_\subdkls \to \sset_{(\infty,1)}$ is a left Quillen equivalence.

\begin{theorem}\label{theorem comparison with dkls}
The restriction functor $k^* \colon \cset_{\ell,1} \to \cset'_\subdkls$ is a right Quillen equivalence.
\end{theorem}

\begin{proof}
For the proof, we show that left Kan extension $k_! \colon \cset'_\subdkls \to \cset_{\ell,1}$ is a left Quillen equivalence.
The triangulation functors are given by sending the object $[1]^n$ to $\Delta[1]^{n}$ in $\sset$, and then extending using that the category of presheaves is the free cocompletion \cite[Theorem 4.51]{Kelly}.
In particular, we have the following commutative diagram,
\[ \begin{tikzcd}
\square' \rar{k} \dar & \square \dar \rar & \sset \\
\cset' \rar{k_!} \ar[urr, bend right=40,"T'"' near end] & \cset \ar[ur,"T"] 
\end{tikzcd} \]
whose vertical morphisms are Yoneda embeddings.
If we knew that $k_!$ was a left Quillen functor, then we would have a diagram
\[ \begin{tikzcd}[column sep=tiny]
\cset'_\subdkls \ar[rr,"k_!"] \ar[dr,"T'"'] & & \cset_{\ell,1} \ar[dl,"T"] \\
& \sset_{(\infty,1)}
\end{tikzcd} \]
of left Quillen functors with $T$ and $T'$ left Quillen equivalences by \cref{cor cubical QE} and \cite[Theorem 6.1]{DohertyKapulkinLindseySattler:CMI1C}.
By two out of three for Quillen equivalences \cite[Corollary 1.3.15]{Hovey:MC}, this would imply that $k_!$ is a left Quillen equivalence as well.

It remains to check that $k_!$ is a left Quillen functor.
But this is automatic, since $T'$ preserves cofibrations and acyclic cofibrations, and $T$ reflects cofibrations and weak equivalences.
\end{proof}

\begin{remark}
It is also true that $k^* \colon \cset_{m,1} \to \cset'_\subdkls$ is a right Quillen equivalence. To establish this, it is enough to show that $k_!$ preserves monomorphisms.
This uses that $\square'$ is an EZ-Reedy category by \cite[Corollary 1.17]{DohertyKapulkinLindseySattler:CMI1C}, so that monomorphisms are generated by boundary inclusions of representables.
These generators are sent to monomorphisms by $k_!$, hence the same is true for all monomorphisms.
\end{remark}

\section{Further applications}

In this section we give model structures, induced from the Joyal and Kan--Quillen model structures on simplicial sets, on several other categories.
We will study three fully faithful strings of adjoint functors of the form $F\lcolon\cN\lrl\sset_{(\infty,\varepsilon)}\rcolon L,R$, and apply \cref{corollary fully faithful string} to obtain new model structures on $\cN$ for $(\infty,\varepsilon)$-categories.

\subsection{Model structures on prederivators}\label{subsec pred}

Let $\catfin$ denote the $2$-category of \emph{homotopically finite categories}, namely those categories whose nerve has only finitely many nondegenerate simplices, and let $\pder$ denote the category of \emph{small prederivators}, namely $2$-functors $\mathbb D\colon\catfin^{\op}\to\cat$, and strict natural transformations. As discussed in \cite[\textsection1]{FKKR}, this category is a locally presentable category of prederivators, as opposed to the more traditional (large and not locally presentable) category of $2$-functors $\mathbb D\colon\cat^{\op}\to\cC AT$.

There is an underlying functor $U\colon\pder\to\sset$, defined by $(U\mathbb D)_n \coloneqq \operatorname{ob} (\mathbb D([n]))$.
It is shown in \cite[\textsection1.15, 1.16]{FKKR} that the functor $U$ admits both a left and a right adjoint, and it is shown in \cite[Proposition~1.18]{FKKR} that the left adjoint is fully faithful.
In particular, there is a fully faithful string of adjoint functors $U\colon\pder\lrl\sset$ between the category of simplicial sets and the category of small prederivators.

From \cref{corollary fully faithful string}, we obtain the following, which endows $\pder$ with two Quillen equivalent model structures for $(\infty,\varepsilon)$-categories for any fixed $\varepsilon=0,1$.
   
\begin{prop}
Let $\varepsilon=0,1$.
The category $\pder$ admits both the left-induced model structure $\pder_{\ell,\varepsilon}$ and the right-induced model structure  $\pder_{r,\varepsilon}$ along $U\colon\pder\to\sset_{(\infty,\varepsilon)}$. Further, the left diagram below is a diagram of left Quillen equivalences, and the right diagram below is a diagram of right Quillen equivalences.
\[ \begin{tikzcd}[column sep=tiny]
\pder_{r,\varepsilon}\ar[rr,"\id_{\pder}"] \ar[dr,"U"'] & & \pder_{\ell,\varepsilon} \ar[dl, "U"] 
&
\pder_{r,\varepsilon} \ar[dr,"U"'] & & \pder_{\ell,\varepsilon} \ar[dl, "U"]  \ar[ll,"\id_{\pder}"']
\\
& \sset_{(\infty,\varepsilon)} & & & \sset_{(\infty,\varepsilon)}
\end{tikzcd} \]
\end{prop}

The model structure $\pder_{r,1}$ is the one considered in \cite[\textsection3]{FKKR}, while the others are new.
   
  \subsection{Model structures on marked simplicial sets}\label{subsec marked}

   Let $\msset$ denote the category of marked simplicial sets, namely simplicial sets endowed with a specified set of \emph{marked} $1$-simplices, and maps that preserve the marking, as in \cite[\textsection3.1]{htt}.
   
   There is an underlying functor $U\colon\msset\to\sset$ which just forgets the marking.
   This functor admits both a left adjoint and a right adjoint which are given by the minimal and maximal marking respectively, $(-)^{\flat},(-)^{\sharp}\colon\sset\to\msset$. 
   The minimal and maximal marking functors are fully faithful, since the unit of the adjunction $(-)^{\flat}\dashv U$ is an identity. 
In particular, there is a fully faithful string of adjoint functors
\[U\lcolon \msset\lrl\sset\rcolon (-)^{\flat},(-)^{\sharp}\]
    between the category of simplicial sets and the category of marked simplicial sets.

From \cref{corollary fully faithful string}, we obtain the following, which endows $\msset$ with two Quillen equivalent model structures for $(\infty,\varepsilon)$-categories for any fixed $\varepsilon=0,1$.

\begin{prop}
Let $\varepsilon=0,1$.
The category $\msset$ admits both the left-induced model structure $\sset^{+}_{\ell,\varepsilon}$ and the right-induced model structure $\sset^{+}_{r,\varepsilon}$ along $U\colon\msset\to\sset_{(\infty,\varepsilon)}$. Further, the left diagram below is a diagram of left Quillen equivalences, and the right diagram below is a diagram of right Quillen equivalences.
\[ \begin{tikzcd}[column sep=tiny]
\sset^{+}_{r,\varepsilon}\ar[rr,"\id_{\msset}"] \ar[dr,"U"'] & & \sset^{+}_{\ell,\varepsilon} \ar[dl, "U"] 
&
\sset^{+}_{r,\varepsilon} \ar[dr,"U"'] & & \sset^{+}_{\ell,\varepsilon} \ar[dl, "U"]  \ar[ll,"\id_{\msset}"']
\\
& \sset_{(\infty,\varepsilon)} & & & \sset_{(\infty,\varepsilon)}
\end{tikzcd} \]
\end{prop}

These model structures on $\msset$ are all new, and seemingly different from the model structure on $\msset$ that models $\infty$-categories constructed by Lurie in \cite{htt}. 

\begin{rmk}\label{lurie model structure remark}
Recall from \cite[\S3.1.3]{htt} the \emph{Cartesian model structure} on $\msset$, which we denote by $\msset_\sublurie$ where the cofibrations are the monomorphisms and the fibrant objects are the naturally marked quasi-categories.
Lurie shows in \cite[Proposition 3.1.5.3]{htt} that the functor $U\colon\msset_\sublurie\to\sset_{(\infty,1)}$
is a right Quillen equivalence. 
Although the model structures $\msset_{\ell,1}$ and $\msset_{r,1}$ do not seem to be comparable with $\msset_\sublurie$ via the identity functor, we have the following composable chain of Quillen equivalences.
\[ \begin{tikzcd}
\msset_{r,1} \rar[shift left, "\id"] & 
\msset_{\ell,1} \lar[shift left, "\id"] \rar[shift left, "U"] &  
\sset_{(\infty,1)} \lar[shift left, "(-)^{\sharp}"] \rar[shift left, "(-)^{\flat}"] & 
\msset_\sublurie \lar[shift left, "U"]
\end{tikzcd} \]
\end{rmk}

\begin{remark}
There is an adjoint string $\mathcal St_{\leq n} \lrl \mathcal St_{(\infty,n)}$ given in \cite[\S2.5]{EmilyNotes}, where $\mathcal St_{(\infty,n)}$ is the category of stratified simplicial sets equipped with the model structure for $(\infty,n)$-categories \cite[4.25]{EmilyNotes}, and $\mathcal St_{\leq n}$ are objects marked only through dimension $n$.
This is an idempotent and homotopy idempotent adjoint string, though it is not fully faithful.
The left- and right-induced model structures on $\mathcal St_{\leq n}$ exist and coincide, and the inclusion is a right Quillen equivalence.
Moreover, when $n=0$ we have $\mathcal St_{\leq 0} = \sset_{(\infty,0)}$ and when $n=1$ we have $\mathcal St_{\leq 1} = \msset_\sublurie$.
\end{remark}
 
  \subsection{Model structures on bisimplicial sets} \label{subsec bisimp}
   
   Let $\ssset$ denote the category of bisimplicial sets, and $i_1^*\colon\ssset\to\sset$ the  zeroth row functor, defined by taking the first row of a bisimplicial set as in \cite[\textsection4]{JT}.
   The functor $i_1^*$ can be seen as the functor induced by
   the fully faithful inclusion $i_1\colon \Delta\hookrightarrow\Delta\times\Delta$, given by $[n]\mapsto[n]\times[0]$. In particular, $i_1^*$ admits a left and a right adjoint, $(i_1)_!,(i_1)_*\colon\sset\to\ssset$ obtained as left and right Kan extension along the fully faithful functor $i_1$, and they are therefore fully faithful. In particular, there is a fully faithful string of adjoint functors
\[i_1^*\lcolon \ssset\lrl\sset\rcolon (i_1)_!,(i_1)_*\]
    between the category of simplicial sets and the category of bisimplicial sets.

From \cref{corollary fully faithful string}, we obtain the following, which endows $\ssset$ with two Quillen equivalent model structures for $(\infty,\varepsilon)$-categories for any fixed $\varepsilon=0,1$.

\begin{prop}
Let $\varepsilon=0,1$.
The category $\ssset$ admits both the left-induced model structure $\ssset_{\ell,\varepsilon}$ and the right-induced model structure $\ssset_{r,\varepsilon}$ along $i_1^* \colon\ssset\to\sset_{(\infty,\varepsilon)}$. Further, the left diagram below is a diagram of left Quillen equivalences, and the right diagram below is a diagram of right Quillen equivalences.
\[ \begin{tikzcd}[column sep=tiny]
\ssset_{r,\varepsilon}\ar[rr,"\id_{\ssset}"] \ar[dr,"i_1^*"'] & & \ssset_{\ell,\varepsilon} \ar[dl, "i_1^*"] 
&
\ssset_{r,\varepsilon} \ar[dr,"i_1^*"'] & & \ssset_{\ell,\varepsilon} \ar[dl, "i_1^*"]  \ar[ll,"\id_{\ssset}"']
\\
& \sset_{(\infty,\varepsilon)} & & & \sset_{(\infty,\varepsilon)}
\end{tikzcd} \]
\end{prop}

These model structures on $\ssset$ are all new, and seemingly different from the model structure on $\ssset$ that models $\infty$-categories constructed by Rezk in \cite{rezkhomotopy}.

\begin{rmk}\label{Rezk model structure remark}
Recall from \cite[Theorem 7.2]{rezkhomotopy} the \emph{complete Segal space model structure} on $\ssset$, which we denote by $\ssset_\subrezk$ where the cofibrations are the monomorphisms and the fibrant objects are the (injectively fibrant) complete Segal spaces.
Joyal--Tierney show in \cite[Theorem 4.11]{JT} that the functor $i_1^*\colon\ssset_\subrezk\to\sset_{(\infty,1)}$
is a right Quillen equivalence. 
Although the model structures $\ssset_{\ell,1}$ and $\ssset_{r,1}$ do not seem to be comparable with $\ssset_\subrezk$ via the identity functor, we have the following composable chain of Quillen equivalences.
\[ \begin{tikzcd}
\ssset_{r,1} \rar[shift left, "\id"] & 
\ssset_{\ell,1} \lar[shift left, "\id"] \rar[shift left, "i_1^*"] &  
\sset_{(\infty,1)} \lar[shift left, "(i_1)_*"] \rar[shift left, "(i_1)_!"] & 
\ssset_\subrezk \lar[shift left, "i_1^*"]
\end{tikzcd} \]
\end{rmk}

\bibliographystyle{amsalpha}
\bibliography{ref}
\end{document}